\documentclass [11pt,twoside,a4paper]{article}
\usepackage{amsfonts}
\usepackage{amsthm}
\usepackage{amsmath}
\usepackage{amstext}
\usepackage{amssymb}
\usepackage{mathrsfs}
\usepackage{amscd}
\usepackage{xypic}

\numberwithin{equation}{section}

 \def\qed{\hfill$\Box$\medskip}
 \newtheorem{theorem}{Theorem}[section]
 \newtheorem{lemma}[theorem]{Lemma}

 \newtheorem{proposition}[theorem]{Proposition}

 \newtheorem{remark}[theorem]{Remark}

 \def\<{\langle}\def\>{\rangle}

 \def\proof{\noindent{\bf Proof.~}}

 \def\beqlb{\begin{eqnarray}}\def\eeqlb{\end{eqnarray}}
 \def\beqnn{\begin{eqnarray*}}\def\eeqnn{\end{eqnarray*}}

 \def\<{\langle}\def\>{\rangle}

 \def\proof{\noindent{\bf Proof.~}}

 \def\beqlb{\begin{eqnarray}}\def\eeqlb{\end{eqnarray}}
 \def\beqnn{\begin{eqnarray*}}\def\eeqnn{\end{eqnarray*}}

\begin{document}

\bigskip\bigskip

\centerline{\LARGE\bf A note on the critical barrier for the survival of $\alpha-$stable}
\medskip

\centerline{\LARGE\bf branching random walk with absorption }
\medskip
\centerline{
}
\medskip

\bigskip

\centerline{Jingning Liu and Mei
Zhang\,\footnote{Corresponding author, meizhang@bnu.edu.cn }}

\bigskip\bigskip

\centerline{School of Mathematical Sciences  }
 \centerline{Laboratory of Mathematics and Complex Systems }
 \centerline{ Beijing Normal University }

 \centerline{Beijing 100875, People's Republic of China}

 \centerline{E-mails:\;{\tt
liujingning@mail.bnu.edu.cn
  and
meizhang@bnu.edu.cn}}

\medskip
{\narrower{\narrower

\noindent{\textbf{\it Abstract.}} We consider a branching random walk with an absorbing barrier, where the step of the associated one-dimensional random walk is in the domain of attraction of an $\alpha$-stable law with $1<\alpha<2$. We shall prove that there is a barrier $an^{\frac{1}{1+\alpha}}$ and a critical value $a_\alpha$ such that if $a<a_\alpha$, then the process dies; if $a>a_\alpha$, then the process survives. The results generalize previous results in literature  for the case $\alpha=2$.

\medskip


\bigskip

\medskip

{\narrower

\section{Introduction}

We consider a discrete-time one-dimensional branching random walk. It starts with an initial ancestor particle located at the origin. At time $1$, the particle dies, producing a certain number of new particles. These new particles are positioned according to the distribution of the point process $\Theta$. At time $2$, the above particles die, producing  new particles positioned (with respect to the birth place) according to  $\Theta$, and the process goes on with the same mechanism. We assume the particles produce new particles independently of each other at the same generation and of everything up to that generation. This system can be seen as a branching tree $\mathcal{T}$ with the origin as the root.

For each vertex $x$ on $\mathcal{T}$, we denote its position by $V(x)$. The family of the random variables $(V(x))$ is usually referred to as a branching random walk (Biggins~\cite{bi}).

Throughout the paper, we assume: %(in the sense of \cite{BK}):
\begin{align}\label{C:1.1}
\mathbb{E}\Big(\sum_{|x|=1}1\Big)>1, \;\; \mathbb{E}\Big(\sum_{|x|=1}e^{-V(x)}\Big)=1,\;\; \mathbb{E}\Big(\sum_{|x|=1}V(x)e^{-V(x)}\Big)=0,
\end{align}
where $|x|$ denotes the generation of $x$. This assumption is referred to in the literature as the boundary case; see for example Biggins and Kyprianou \cite{BK}. Every branching random walk satisfying certain mild integrability assumptions can be reduced to this case by some renormalization; see Jaffuel \cite{j} for more details. Note that \eqref{C:1.1} implies  $\mathcal{T} $ is a super-critical Galton-Watson tree.

Denote $\mathbb{N}=\{0,1,2,\cdots\}$ and $\mathbb{N^*}=\{1,2,\cdots\}$.
We define a ``barrier" by a function $\varphi: \mathbb{N}\rightarrow\mathbb{R}$  and consider  the branching random walk with absorption: On $\mathcal{T}$, all the individuals $x$ such that $V(x)>\varphi(|x|)$, i.e. born above the barrier, are immediately removed and do not reproduce.

A natural question is whether the process survives or not.  Kesten \cite{Ke}, Derrida and Simon \cite{DeS1,DeS2}, Harris J. and Harris S. \cite{DH} have studied the continuous analog of this process, the branching Brownian motion with absorption. Biggins et al~\cite{B7}  solved the corresponding question on the linear barriers. Under certain conditions (see (\ref{C:1.1})--(\ref{C:1.3})), Jaffuel \cite{j} refined above result by considering a more general barrier.  He found a barrier $an^\frac{1}{3}$ and a critical value $\hat{a}$: the process dies when $a<\hat{a}$ and survives when $a>\hat{a}$.

Before stating the results in literatures, we introduce some notation. We denote by $   u_i$ the ancestor of $   u$ in generation $i$ and $\mathcal{T}_n:=\{   u\in\mathcal{T}:|   u|=n\}$ the population at time $n$. And we say $x<y$ iff individual $x$ is an ancestor of individual $y$.
Define an infinite path $   u$ through $\mathcal{T}$ as a sequence of individuals $   u=(   u_i)_{i\in\mathbb{N}}$ such that
$$\forall\,i\in\mathbb{N},\;|   u_i|=i,\; \;   u_i<   u_{i+1},
$$
and denote their collection by $\mathcal{T}_\infty$. For $A\subset \mathcal{T}$, $\# A $ denotes the number of individuals in $A$.

\begin{theorem}\label{T:1.2}{\rm{\textbf{(Biggins et. al. \cite{B7})}}.}
Under condition \eqref{C:1.1}, we have
$$\mathbb{P}\big(\exists\,   u\in\mathcal{T}_\infty,\,\forall\,i\geq1,\,V(   u_i)\leq i\varepsilon\big)
\begin{cases}
=0,& \text{if $\varepsilon\leq0$},\\
>0,& \text{if $\varepsilon>0$}.
\end{cases}$$
\end{theorem}

To present the result in Jaffuel~\cite{j}, we need more conditions.

\begin{align}
&(\romannumeral1)\quad \exists\, \delta>0,\mathbb{E}(\#\mathcal{T}_1^{1+\delta})<+\infty,\label{C:1.21} \\
& (\romannumeral2)\quad \exists\, \varrho>0,\,\mathbb{E}\Big(\sum_{|   u|=1}e^{-(1+\varrho)V{(   u)}}\Big)<+\infty, \label{C:1.2}\\
&(\romannumeral3)\quad \sigma^2:=\mathbb{E}\Big(\sum_{|   u|=1}V{(   u)}^2e^{-V{(   u)}}\Big)<+\infty.\label{C:1.3}
\end{align}

Jaffuel~\cite{j} refined Theorem~\ref{T:1.2} by replacing the linear barrier $i\varepsilon$ with a barrier $\varphi(i):=ai^{1/3}$.

\begin{theorem}\label{Jaffuel} {\rm{\textbf{(Jaffuel \cite{j})}}.}
Let $\hat{a}=\frac{3}{2}(3\pi^2\sigma^2)^{1/3}$. Assuming \eqref{C:1.1}-- \eqref{C:1.3}, we have

  \beqnn \mathbb{P}\big(\exists\,   u\in\mathcal{T}_\infty,\,\forall\,i\geq1,\,V(   u_i)\leq ai^{1/3}\big)
\begin{cases}
>0,& \text{if $a>\hat{a}$},\\
=0,& \text{if $a<\hat{a}.$}
\end{cases}\eeqnn
\end{theorem}

The aim of the present paper is  to replace condition \eqref{C:1.3} by
\begin{align}\label{C:1.4}
  \mathbb{E} \Big(\sum_{|x|=1}\mathbf{1}_{\{|V(x)|\geq y\}} e^{-V(x)}\Big)\sim \frac{c}{y^{\alpha}},\quad y\to +\infty,
\end{align}
where $c\in(0,\infty)$ and~$\alpha\in(1,2]$. Actually,  (\ref{C:1.4}) turns out to be
\begin{align}\label{C:1.4a}
  \mathbb{E} \Big(\sum_{|x|=1}\mathbf{1}_{\{V(x)\geq y\}} e^{-V(x)}\Big)\sim \frac{c}{y^{\alpha}},\quad y\to +\infty
\end{align}
under (\ref{C:1.2}). Now, if we  define a random variable $X$  by
\begin{align}\label{X}
\mathbb{P}(X\leq x)=\mathbb{E}\Big(\sum_{|   u|=1}1_{\{V(   u)\leq x\}}e^{-V(   u)}\Big),\quad x\in \mathbb{R},
\end{align}
 then we shall see from Lemma~\ref{L:2.1} and \eqref{C:1.4} that
$$
\mathbb{P}(X>x)\sim cx^{-\alpha},\;\;\mathbb{P}(X<-x)= o(x^{-\alpha}),\;\; x\to +\infty,
$$
which means that
$X$ is in the domain of attraction of a strictly $\alpha$-stable random variable $Y$  with characteristic function of the form
\begin{align}\label{Y}
G_{\alpha}(t)=\exp\{-c_0|t|^\alpha(1-i\frac{t}{|t|} \tan \frac{\pi\alpha}{2})\},\;c_0>0,\quad \alpha\in(1,2]. \end{align} Denote by $(Y_t,\,t\in[0,1])$ the strictly $\alpha$-stable L\'{e}vy process such that $Y_1$ has the same law as $Y$. Under \eqref{C:1.1} and \eqref{C:1.4}, we call $(V(x))$  a stable branching random walk. We mention that the convergence of derivative martingale and additive martingale for stable branching random walk  are studied in a recent paper \cite{HLZ}; The asymptotic behavior of the position of $N$-branching random walk with $\alpha$-stable spine has been studied in Mallein~\cite{Mal}.

  Theorem~\ref{T:1.1} and Proposition~\ref{P:1.3} are our main results.

\begin{theorem}\label{T:1.1}
Let $a_\alpha=(1 + \alpha^{-1})(\alpha(1+\alpha)C_*)^{\frac{1}{1+\alpha}}$, where \beqlb\label{C*}C_*=C_*(\alpha):=-\lim_{t\rightarrow\infty}\frac{1}{t}\log\mathbb{P}\big(|Y_s|\leq \frac{1}{2},\,s\leq t\big)\in(0,+\infty).\eeqlb
 Assume \eqref{C:1.1}--\eqref{C:1.2} and \eqref{C:1.4}. Then
\beqlb\mathbb{P}\big(\exists\,   u\in\mathcal{T}_\infty,\,\forall\,i\geq1,\,V(   u_i)\leq ai^{\frac{1}{1+\alpha}}\big)
  \begin{cases}
>0,& \text{if $a>a_\alpha$},\quad\mbox{\rm{(lower bound)}} \\
=0,& \text{if $a<a_\alpha$}, \,\quad\mbox{\rm{(upper bound)}}
\end{cases} \eeqlb

\end{theorem}

 \begin{remark} We observe that if $\alpha=2$, then condition (\ref{C:1.2}) and (\ref{C:1.4}) reduce to (\ref{C:1.2}) and (\ref{C:1.3}). In this case, $C_*=C_*(2)=\frac{\pi^2\sigma^2}{2}$, which coincides with Theorem~\ref{Jaffuel}.
\end{remark}

 By taking derivative and discussing monotonicity of the function $$
 f(x)=x+\frac{ 1+\alpha }{x^\alpha}C_*,\quad x>0.$$
 we can see that $$
  f'(\frac{\alpha a_\alpha}{ 1+\alpha})=0,$$
 and $\min_x f(x)=  a_\alpha$.
   If $a>a_\alpha$,  then the equation $a=x+\frac{ 1+\alpha }{x^\alpha}C_*$ has two solutions in $x$. Let $r_a$ be the larger solution. Clearly, $r_a>\frac{\alpha a_\alpha}{ 1+\alpha }$.

\begin{proposition}\label{P:1.3}
For $a>a_\alpha$, $\varepsilon\in (0, r_a)$, $N\in \mathbb{N^*}$, define
 $$
B_N=\{\forall k\ge1, \#\{   u\in\mathcal{T}_{N^k}:\,\forall\,i\leq N^k,\,(a-r_a)i^{\frac{1}{1+\alpha}}\leq V(   u_i)\leq ai^{\frac{1}{1+\alpha}}\}\geq
\exp(\frac{1}{2}N^{\frac{k}{1+\alpha}}(r_a-\varepsilon))\}.
$$
Then for sufficiently large $N$,  $\mathbb{P}(B_N)>0$.

\end{proposition}

 The remainder of the paper is organized as follows. In section 2,  we  prove several lemmas on the one-dimensional random walk associated with $(V(x))$, which will be used in the proofs of Theorem~\ref{T:1.1} and Proposition~\ref{P:1.3}. In  section 3, we prove  Proposition \ref{P:1.3} and the lower bound for the survival probability in Theorem~\ref{T:1.1}. The upper bound in Theorem~\ref{T:1.1} are discussed in section 4. Our technical routes and proofs are based on Mallein~\cite{Mal}, A\"{\i}d\'{e}kon and Jaffuel~\cite{AiJa} and Jaffuel \cite{j}.

\section{Small deviations estimate and variations}

Let $(X_i)_{i\in\mathbb{N}}$ be an i.i.d. sequence of copies of $X$ defined by (\ref{X}). Let $S_0=0$, and  for any $n\ge 1 $, $$S_n:=\sum_{i=1}^nX_i. $$ $S$ is then a mean-zero heavy-tailed random walk starting from the origin.

\begin{lemma}\label{L:2.1}\rm{ \textbf{(many-to-one lemma, Biggins and Kyprianou \cite{BK4})}}. \emph{For any $n\geq1$ and any measurable function F: $\mathbb{R}^n\rightarrow[0,+\infty)$,
$$
\mathbb{E}\Big(\sum_{|   u|=n}e^{-V(   u)}F\big(V(   u_i),\,1\leq i\leq n\big)\Big)
=\mathbb{E}\big(F(S_i),\,1\leq i\leq n\big).
$$}
\end{lemma}

{  Our proof of the lower bound { of Theorem~\ref{T:1.1}}  requires the following bivariate version of the many-to-one lemma.}
\begin{lemma}\label{L:2.2}\rm{\textbf{(Gantert, Hu and Shi \cite{GHS})}} \emph{Suppose $X$ is a random variable defined by (\ref{X}). Let $\upsilon$ be a random variable taking values in $\mathbb{N}^*$ such that for any nonnegative measurable function $f$,
$$
\mathbb{E}\big(f(X,\upsilon)\big)=\mathbb{E}\Big(\sum_{|   u|=1}e^{-V(   u)}f(V(   u),\#
\mathcal{T}_1)\Big).
$$
Let $n\geq1$ and $(X_i,\upsilon_i)_{1\leq i\leq n}$ be i.i.d. copies of $(X,\upsilon)$. Then for any measurable function $F: (\mathbb{R}\times\mathbb{N}^*)^n\rightarrow[0,+\infty)$, it holds that
$$
\mathbb{E}\Big(\sum_{|   u|=n}e^{-V(   u)}F(V(   u_i),\,\#\Gamma(   u_{i-1}),1\leq i\leq n)\Big)=\mathbb{E}\big(F(S_i,\upsilon_i,1\leq i\leq n)\big),
$$
where $\Gamma(   u):=\{  v\in\mathcal{T}:|  v|=|   u|+1,  v>   u\}$.}
\end{lemma}

 Let $\mathcal{F}[0,1]$ (respectively, $\mathcal{C}[0,1]$) be the set of functions (respectively, continuous functions) from $[0,1]$ to $\mathbb{R}$. For  $z\in\mathbb{R}$, we denote by $\mathbb{P}_z$ the probability associated with the branching random walk $(V(x))$ starting from $z$, and $\mathbb{E}_z$ the corresponding expectation. We breviate $\mathbb{P}_0$ by $\mathbb{P}$.

 In the following Theorem \ref{T:2.1} and Lemma~\ref{L:2.9}, let $(  c_n)$ be a sequence of positive real numbers such that
$$
\lim_{n\rightarrow\infty}   c_n=+\infty, \;\;\lim_{n\rightarrow\infty}\frac{  c_n}{n^{1/\alpha}}=0.
$$

\begin{theorem}\rm{\textbf{(Mogul'ski\v{\i}}} \cite{Mog})\label{T:2.1}
Let $f,g\in\mathcal{C}[0,1],$ with $f<g$ and $f(0)<0<g(0)$. We have
\begin{align}
\lim_{n\rightarrow \infty} \frac{  c_n^\alpha}{n}\log \mathbb{P}\bigg(\frac{S_j}{  c_n}\in\Big[f\Big(\frac{j}{n}\Big),g\Big(\frac{j}{n}\Big)\Big],\,0\leq j\leq n\bigg)=-C_*\int_0^1\frac{ds}{(g(s)-f(s))^\alpha}, \nonumber
\end{align}
where $C_*$ is defined by (\ref{C*}).
\end{theorem}

The following three lemmas are some more sophisticated versions of  above theorem. For the proofs of them,  we shall borrow some ideas from A\"{\i}d\'{e}kon and Jaffuel~\cite{AiJa}, which discussed the case $\alpha=2$.

\begin{lemma}\label{L:2.9}
Let $f,g\in\mathcal{C}[0,1],$ with $f<g$ and $f(0)<0<g(0)$. For any sequences  $(f_n)$ and $(g_n)$ of $\mathcal{F}[0,1]$ such that $\|f_n-f\|_{\infty}\rightarrow0$ and $\|g_n-g\|_{\infty}\rightarrow0$ as $n\rightarrow \infty,$ we have
\begin{align}
\lim_{n\rightarrow \infty} \frac{  c_n^\alpha}{n}\log  \mathbb{P}\bigg(\frac{S_j}{  c_n}\in\Big[f_n\Big(\frac{j}{n}\Big),g_n\Big(\frac{j}{n}\Big)\Big],\,0\leq j\leq n\bigg)=-C_*\int_0^1\frac{ds}{(g(s)-f(s))^\alpha}, \nonumber
\end{align}
where $C_*$ is defined by (\ref{C*}).\end{lemma}

\begin{proof}
Let $0<\,\varepsilon< \frac{1}{2}\min\{\min_{[0,1]}|g-f|, |f(0)|, |g(0)|\}$. We can choose $N\geq1$ s.t. for any $n\geq N,\;\max_{[0,1]}|f_n-f|+\max_{[0,1]}|g_n-g|<\varepsilon$. Then for such $n$, we have
\begin{align}
\Big\{f+\varepsilon<\frac{S_j}{  c_n}<g-\varepsilon\Big\}\subset\Big\{f_n<\frac{S_j}{  c_n}<g_n\Big\}
\subset \Big\{f-\varepsilon<\frac{S_j}{  c_n}<g+\varepsilon\Big\} \nonumber
\end{align}
Applying Theorem \ref{T:2.1}, we get
\begin{align}
& -C_*\int_0^1\frac{ds}{(g-f-2\varepsilon)^\alpha}\nonumber\\&\leq \liminf_{n\rightarrow\infty}
\frac{  c_n^\alpha}{n}\log  \mathbb{P} \bigg(\frac{S_j}{  c_n}\in\Big[f_n\Big(\frac{j}{n}\Big),g_n\Big(\frac{j}{n}\Big)\Big],\,0\leq j\leq n\bigg)\nonumber\\
&\leq \limsup_{n\rightarrow\infty}
\frac{  c_n^\alpha}{n}\log  \mathbb{P} \bigg(\frac{S_j}{  c_n}\in\Big[f_n\Big(\frac{j}{n}\Big),g_n\Big(\frac{j}{n}\Big)\Big],\,0\leq j\leq n\bigg)\nonumber\\
&\leq -C_*\int_0^1\frac{ds}{(g-f+2\varepsilon)^\alpha}.\nonumber
\end{align}
Letting $\varepsilon\rightarrow0$, we complete the proof.\qed
\end{proof}

\begin{lemma}\label{L:2.3}
Let $f,g\in\mathcal{C}[0,1],$ with $f<g$ and $f(0)<0<g(0)$. Let $(f_n)$ and $(g_n)$
be sequences of $\mathcal{F}[0,1]$ such that $\|f_n-f\|_{\infty}\rightarrow0$ and $\|g_n-g\|_{\infty}\rightarrow0$ as $n\rightarrow \infty.$ Let $\beta^*$ and $\gamma^*$ be positive real numbers such that $0\leq \beta^*< \gamma^* \leq1$. Let
$   u^*,  v^*\in\mathbb{R}$ s.t. $f(\beta^*)\leq    u^* <  v^*\leq g(\beta^*)$.
Let $\gamma(n)>\beta(n)$ be the sequences of positive integers, and $(   \mu_n)_n,\,(  \nu_n)_n$ be sequences of reals such that

$$
n^{-\frac1{1+\alpha}}   \mu_n \rightarrow    u^*,\;\;n^{-\frac1{1+\alpha}}  \nu_n \rightarrow  v^*,\;\;
\frac{\beta(n)}{n}\rightarrow\beta^*,\;\;\frac{\gamma(n)}{n}\rightarrow\gamma^*,
$$
and for any $n\geq1$,
$$
f_n(\beta(n)/n)n^\frac1{1+\alpha}\leq    \mu_n\leq   \nu_n\leq g_n(\beta(n)/n)n^\frac1{1+\alpha}, \;0\leq \beta(n)<\gamma(n)\leq n.
$$
We also assume that
\begin{align}
\exists\; M\in\mathbb{N}^*,\; \forall\; m\in\mathbb{N}^*,\;\#\{n: \gamma(n)-\beta(n)=m\}\leq M.  \label{ASS}
\end{align}Then
\begin{align}
&\lim_{n\rightarrow \infty} n^{-\frac{1}{1+\alpha}}\log \Bigg(\sup_{   \mu_n\leq z\leq  \nu_n}  \mathbb{P}_z\bigg(\frac{S_{k-\beta(n)}}{\,n^\frac1{1+\alpha}}\in\Big[f_n\Big(\frac{k}{n}\Big),g_n\Big(\frac{k}{n}\Big)\Big]\;,\,
 \beta(n)< k\leq \gamma(n)\bigg) \Bigg)\nonumber\\
&\leq-C_*\int_{\beta^*}^{\gamma^*}
\frac{ds}{(g(s)-   u^*-f(s)+  v^*)^\alpha}. \nonumber
\end{align}
\end{lemma}

\begin{proof}
Here we set $  c_n=n^\frac1{1+\alpha}$.
%and write $m=\gamma(n)-\beta(n)$.
Notice that for $m\in A:=\{m\in\mathbb{N}^*:\exists\; n\in \mathbb{N}^*,\;\gamma(n)-\beta(n)=m\}$,
\begin{equation}\label{E:2.31}
\begin{array}
{l}\Big\{\forall z\in[   \mu_n,  \nu_n],\;\forall k\leq m,\;   c_nf_n\Big(\frac{\beta(n)+k}{n}\Big)\leq z+S_k \leq c_ng_n\Big(\frac{\beta(n)+k}{n}\Big)\Big\} \\\subset \Big\{\forall k\leq m,\;  c_nf_n\Big(\frac{\beta(n)+k}{n}\Big)-  \nu_n\leq S_k\leq   c_ng_n\Big(\frac{\beta(n)+k}{n}\Big)-   \mu_n\Big\}.
\end{array}
\end{equation}

By  \eqref{ASS}, we can define a surjection $\varphi:\{1,2,\cdots,M\}\times A\longrightarrow \mathbb{N}^*$ such that for any $1\leq l\leq M$ and $ m\in A,$ it holds that $\;m=\gamma(\varphi(l,m))-\beta(\varphi(l,m))$. For such $l,m$,  define

\begin{align}
\widetilde{f}_m(t)=\frac{f_n(\frac{(1-t)\beta(n)+t\gamma(n)}{n})  c_n-  \nu_n}{c_m},\;\;
\widetilde{g}_m(t)=\frac{g_n(\frac{(1-t)\beta(n)+t\gamma(n)}{n})  c_n-   \mu_n}{c_m}\nonumber,
\end{align}
where $n=\varphi(l,m)$. {It is not difficult to see that $n\sim(\gamma^*-\beta^*)^{-\!1}m,$ and $  c_n\sim c_m(\gamma^*-\beta^*)^{-\!\frac{1}{1+\alpha}}$.} Consequently, as $m\rightarrow\infty$,
\begin{align}
&\widetilde{f}_m(t)\rightarrow \widetilde{f}(t)=(\gamma^*-\beta^*)^{-\!\frac{1}{1+\alpha}}(f((1-t)\beta^*+t\gamma^*)-  v^*),\nonumber\\
&\widetilde{g}_m(t)\rightarrow \widetilde{g}(t)=(\gamma^*-\beta^*)^{-\!\frac{1}{1+\alpha}}(g((1-t)\beta^*+t\gamma^*)-   u^*).\nonumber
\end{align}
For each $1\leq l\leq M$,  applying Lemma \ref{L:2.9} (with $f_n$ and $g_n$ replaced by $\widetilde{f}_n$ and $\widetilde{g}_n$) to the right hand side of \eqref{E:2.31}, we obtain as $m\rightarrow\infty$,
\begin{align}
&\log \mathbb{P}\Big(\forall\;z\in[   \mu_n,  \nu_n],\forall\;k\leq m,   c_nf_n\Big(\frac{\beta(n)+k}{n}\Big)\leq z+S_k\leq   c_ng_n\Big(\frac{\beta(n)+k}{n}\Big)\Big)\nonumber\\
&\leq \log \mathbb{P}\Big( \forall\;k\leq m,   c_nf_n\Big(\frac{\beta(n)+k}{n}\Big)-  \nu_n\leq S_k\leq c_ng_n\Big(\frac{\beta(n)+k}{n}\Big)-   \mu_n\Big)\nonumber\\
& \leq \log \mathbb{P}\Big( \forall\;k\leq m, \widetilde{f}_m\Big(\frac{k}{m}\Big)\leq\frac{S_k}{c_m}\leq\widetilde{g}_m\Big(\frac{k}{m}\Big)\Big)\nonumber\\
&=
-m^{1/(1+\alpha)}C_*\int_0^1\frac{ds}{(\widetilde{g}-\widetilde{f})^\alpha}(1+o(1))\quad (\mbox{by Lemma \ref{L:2.9}}) \nonumber\\
& =-(1+o(1))n^\frac1{1+\alpha}C_*\int_{\beta^*}^{\gamma^*}\frac{ds}{(g-f-   u^*+  v^*)^\alpha}\nonumber.
\end{align}
This bound   holds when $n$ runs along the subsequence $(\varphi(l,m))_m $ for each $1\leq l\leq M$,  which covers all the values $n\in\mathbb{N}^*$, and then the proof is finished. \qed
\end{proof}

\begin{lemma}\label{L:2.4}
Let $f,g\in\mathcal{C}[0,1],$ with $f<g$ and $f(0)<0<g(0)$. Let $(f_n)$ and $(g_n)$
be sequences of $\mathcal{F}[0,1]$ such that $\|f_n-f\|_{\infty}\rightarrow0$ and $\|g_n-g\|_{\infty}\rightarrow0$ as $n\rightarrow \infty.$ Let $\beta^*$ and $\gamma^*$ be positive real numbers such that $0\leq \beta^*< \gamma^* \leq1$. Let $0\leq\beta(n)<\gamma(n)\leq n$ be the sequences of positive integers such that

$$
\frac{\beta(n)}{n}\rightarrow\beta^*,\;\;\frac{\gamma(n)}{n}\rightarrow\gamma^*,
$$
and assume \eqref{ASS}. Then
\begin{align}
\limsup_{n\rightarrow \infty}n^{-\frac{1}{1+\alpha}}\log \Bigg(\sup_{z} \; & \mathbb{P}_z\bigg(\frac{S_{k-\beta(n)}}{n^\frac1{1+\alpha}}\in\Big[f_n\Big(\frac{k}{n}\Big),g_n\Big(\frac{k}{n}\Big)\Big]\;,\,
\nonumber\\ &\forall\,\beta(n)< k\leq \gamma(n)\bigg) \Bigg)\leq-C_*\int_{\beta^*}^{\gamma^*}
\frac{ds}{(g(s)-f(s))^\alpha}, \nonumber
\end{align}
where $\sup_z$ is taken over the set   $\{z\in\mathbb{R}|n^{\frac{1}{1+\alpha}}f_n(\frac{\beta(n)}{n})\leq z\leq n^{\frac{1}{1+\alpha}}g_n(\frac{\beta(n)}{n})\}$.
\end{lemma}

\begin{proof}  { Define $$p(z,n):=\mathbb{P}_z\bigg(\frac{S_{k-\beta(n)}}{n^\frac1{1+\alpha}}\in\Big[f_n\Big(\frac{k}{n}\Big),g_n\Big(\frac{k}{n}\Big)\Big]\;,\,
  \forall\,\beta(n)< k\leq \gamma(n)\bigg).$$}
Let $\varepsilon>0,$ and $N$ be an integer such that $N\varepsilon>g(\beta^*)-f(\beta^*)$. We define for $j=0,1,\cdots,N,$
\begin{align}
   \mu^j_n:=n^\frac1{1+\alpha}\frac{f_n(\beta(n)/n)(N - j)+g_n(\beta(n)/n)j}{N}.\nonumber
\end{align}
Observe that
\begin{align}
\sup_{n^\frac1{1+\alpha}f_n(\frac{\beta(n)}{n})\leq z\leq n^\frac1{1+\alpha}g_n(\frac{\beta(n)}{n})} p(z,n)=\max_{0\leq j\leq N-1}\sup_{   \mu^j_n\leq z\leq    u^{j+1}_n} p(z,n). \nonumber
\end{align}
We apply Lemma \ref{L:2.3} $N$ times, with $   \mu_n=   \mu_n^j$ and $  \nu_n=   u^{j+1}_n,\;j=0,1,\cdots,N-1$ and get
\begin{align}
\limsup_{n\rightarrow\infty} n^{-\frac{1}{1+\alpha}}\log\big(\sup_{n^\frac1{1+\alpha}f_n(\frac{\beta(n)}{n})\leq z\leq n^\frac1{1+\alpha}g_n(\frac{\beta(n)}{n})}p(z,n)\big)\leq -C_*\int_{\beta^*}^{\gamma^*}\frac{ds}{(g(s)-f(s)+\varepsilon)^\alpha}. \nonumber
\end{align}
By letting $\varepsilon\rightarrow 0$, we prove the lemma.\qed

\end{proof}

\begin{remark}\label{R:2.1}

Let $\varepsilon>0$. Notice that the probability that $S_n$ stays between $f$ and $g$ is less than the probability that $S_n$ stays between $\widetilde{f}:=f-\varepsilon$ and $\widetilde{g}:=g+\varepsilon$.
We can extend the upper bounds in Lemmas~\ref{L:2.9}--\ref{L:2.4} to functions satisfying $f\leq g$ and $f(0)\leq 0\leq g(0).$
\end{remark}

\begin{theorem}\rm{\textbf{(Prokhorov theorem \cite{Pro})}}\label{T:2.6}
\emph{If $\frac{S_n}{n^{1/\alpha}}$ converges in law to a strictly stable random variable Y, then the process $\Big\{\frac{S_{\lfloor nt\rfloor}}{n^{1/\alpha}},\,t\in[0,1]\Big\}$ converges in law to an $\alpha$-stable L\'{e}vy process $\big\{Y_t,t\in[0,1]\big\}$ in $\mathcal{D}([0,1])$ equipped with the Skorokhod topology such that $Y_1$ has the same law as Y.}
\end{theorem}

Using an adjustment of the original proof of Mogul'ski\v{\i},  similarly to \cite{Mal}, we have two  estimates for an  enriched random walk.  Recall that $(\upsilon_j)$ defined in Lemma \ref{L:2.2} is a sequence of $N^*$-valued i.i.d. random variables.
%Recall that $(S_j,\upsilon_j)$ is  the pair of random variables defined in Lemma \ref{L:2.2}.}

\begin{lemma}\label{L:2.9r}
Let $f,g\in\mathcal{C}[0,1],$ with $f < g$ and $f(0) < 0 < g(0)$. We set $E_k^{(n)}=\{\upsilon_j\leq \exp\{n^{1/\beta}\},j\leq k\}$  for some $\beta>0$ and  assume that
\begin{align}
\lim_{n\rightarrow\infty}n^{\alpha/(1+\alpha)}\mathbb{P}\big(\upsilon_1\geq \exp\{n^{1/\beta}\}\big)=0 .\nonumber
\end{align}
For any $f(0) < x < y < g(0)$, we have
\begin{align}
&\lim_{n\rightarrow \infty}  n^{-\frac1{1+\alpha}} \log \inf_{z\in[x,y]}  \mathbb{P}_{zn^{\frac1{1+\alpha}}}\bigg(\frac{S_j}{n^{\frac1{1+\alpha}}}\in\Big[f\Big(\frac{j}{n}\Big),g\Big(\frac{j}{n}\Big)\Big],\,0\leq j\leq n,\;E_n^{(n)}\bigg)\nonumber\\
&=-C_*\int_0^1\frac{ds}{(g(s)-f(s))^\alpha}. \nonumber
\end{align}
Moreover, for $b>0$,
\begin{align}
\liminf_{n\rightarrow \infty} n^{-\frac1{1+\alpha}}\log \inf_{z\in[x,y]}  & \mathbb{P}_{zn^{\frac1{1+\alpha}}}\bigg(\frac{S_n}{n^{\frac1{1+\alpha}}}\in[g(1)-b,g(1)],\;\nonumber\\
&\frac{S_j}{n^{\frac1{1+\alpha}}}\in\Big[f\Big(\frac{j}{n}\Big),g\Big(\frac{j}{n}\Big)\Big],\,0\leq j\leq n,\;E_n^{(n)}\bigg)\geq-C_*\int_0^1\frac{ds}{(g(s)-f(s))^\alpha}. \nonumber
\end{align}
\end{lemma}

\begin{proof}
In the proof of \cite[Lemma 2.6]{Mal}, replace $E_n$ by $E_n^{(n)}$ and let $  c_n=n^{\frac1{1+\alpha}}$, $r_n=\lfloor A  c_n^\alpha\rfloor$ with $A>0$. Then    with the help of our Lemma~\ref{L:2.4} and Theorem~\ref{T:2.6},   we can go along the line of   \cite[Lemma 2.6]{Mal}  to get the proof. The details are omitted.

\qed
\end{proof}

Replacing Lemma~\ref{L:2.9} by Lemma~\ref{L:2.9r} in the proofs of Lemmas~\ref{L:2.3} and \ref{L:2.4}, we arrive at
\begin{lemma}\label{L:2.8}
Let $f,g\in\mathcal{C}[0,1],$ with $f<g$ and $f(0)<0<g(0)$. Let $(f_n)$ and $(g_n)$
be sequences of $\mathcal{F}[0,1]$ such that $\|f_n-f\|_{\infty}\rightarrow0$ and $\|g_n-g\|_{\infty}\rightarrow0$ as $n\rightarrow \infty.$ Let $\beta^*$ and $\gamma^*$ be positive real numbers such that $0\leq \beta^*< \gamma^* \leq1$. Let $0\leq\beta(n)<\gamma(n)\leq n$ be the sequences of positive integers such that:

$$
\frac{\beta(n)}{n}\rightarrow\beta^*,\;\;\frac{\gamma(n)}{n}\rightarrow\gamma^*,
$$
and assume \eqref{ASS}.  Suppose that $\{\upsilon_j\}$ is defined as in Lemma \ref{L:2.2}.
We set $E_k^{(n)}=\{\upsilon_j\leq \exp\{n^{1/\beta}\},j\leq k\}$ for some $\beta>1+\alpha$.
Then for $b>0$,
\begin{align}
\liminf_{n\rightarrow \infty}  n^{-\frac1{1+\alpha}} \log \inf_{z} \; & \mathbb{P}_{z}\bigg(\frac{S_{\gamma(n)-\beta(n)}}{n^{\frac1{1+\alpha}}}\in\Big[g
\Big(\frac{\gamma(n)}{n}\Big)-b,g\Big(\frac{\gamma(n)}{n}\Big)\Big],\;\nonumber\\
&\frac{S_{j-\beta(n)}}{n^{\frac1{1+\alpha}}}\in\Big[f\Big(\frac{j}{n}\Big),g\Big(\frac{j}{n}\Big)\Big],\,\beta(n)< j\leq \gamma(n),\;E_n^{(n)}\bigg)\geq-C_*\int_{\beta^*}^{\gamma^*}\frac{ds}{(g(s)-f(s))^\alpha}, \nonumber
\end{align}
{where the $\inf_z$ is taken over the set $\{z\in\mathbb{R}\,|\,f\big(\frac{\beta(n)}{n}\big)n^{\frac1{1+\alpha}}\leq z\leq g\big(\frac{\beta(n)}{n}\big)n^{\frac1{1+\alpha}}\}$.}
\end{lemma}

The following lemma will be used to get the lower bound  in Theorem~\ref{T:1.1}.
\begin{lemma}\label{L:2.11}
Let $f,g\in\mathcal{C}[0,1],$ with $f < g$ and $f(0) < 0 = g(0)$. Then there are $M\geq1$ and $\varepsilon_1>0$ such that
 $$
\lim_{\varepsilon_2\rightarrow0}\liminf_{n\rightarrow\infty} n^{-\frac1{1+\alpha}}
\log P_n(M,\varepsilon_1,\varepsilon_2)=0,
$$
where
\begin{align}
&P_n(M,\varepsilon_1,\varepsilon_2) \nonumber\\&=\mathbb{P}\Big(\exists\,    u\in\mathcal{T}_k, \forall\,i<k,\,\#\Gamma(   u_i)\leq M,\,f\Big(\frac{i}{n}\Big)\leq  \frac{V(   u_i)}{n^{\frac1{1+\alpha}}}\leq g\Big(\frac{i}{n}\Big),
-M\varepsilon_2\leq \frac{V(   u_k)}{n^{\frac1{1+\alpha}}}\leq -\varepsilon_1\varepsilon_2\Big),  \nonumber
\end{align}
with $k:=\lfloor \varepsilon_2n^{\frac1{1+\alpha}}\rfloor$.
\end{lemma}
\proof The proof is essentially similar to Jaffuel~\cite[Lemma 2.8]{j} for the case $\alpha=2$, so we omit it.\qed

\section{Lower bound for the survival probability}

In this section we prove  Proposition \ref{P:1.3} and the lower bound for the survival probability in Theorem~\ref{T:1.1}.

We consider the population surviving below the barrier $i\longmapsto ai^{\frac1{1+\alpha}}$: any individual born above the barrier would be removed and do not reproduce.

 Suppose  $\lambda>0$ such that $e^{\lambda}\in\mathbb{N}$.
%and $(  v_k)_{k\geq1}$ be a sequence of } positive numbers.
For any $k\in\mathbb{N}$, we pick a particle $z$ at position $V(z)$ in generation $e^{\lambda k}$, and denote by $Y_k(z)$ the number of descendants it eventually has in generation $e^{\lambda(k+1)}$. Instead of $z$, we pick another particle $\tilde{z}$ in the same generation $e^{\lambda k}$ but positioned on the barrier at $V(\tilde{z}):=ae^{\frac{\lambda k}{1+\alpha}}\geq V(z)$, and suppose the number and displacements of the descendants of $\tilde{z}$ are exactly the same as those of $z$. Clearly, the descendants of $\tilde{z}$ are more likely to cross the barrier and be killed, hence, if we denote  the number of its descendants by $Y_k(\tilde{z})$, then $Y_k(\tilde{z})\leq Y_k(z)$.}

We here add a second absorbing barrier $i\longmapsto (a - b)i^{\frac{1}{1+\alpha}}$ for some $0<b<a$ and kill any descendant of $\tilde{z}$ born below it.
We obtain that, almost surely, $$Z_k\leq Y_k(\tilde{z})\leq Y_k(z), $$ where
$$
Z_k:=\#\{\tilde{   u}\in \mathcal{T}_{e^{\lambda(k+1)}},\;\tilde{   u}>\tilde{z},\;\forall\,e^{\lambda k}<i\leq e^{\lambda(k+1)},\;(a-b)i^{\frac1{1+\alpha}}\leq V(\tilde{   u}_i)\leq ai^{\frac1{1+\alpha}}\}.
$$
Clearly, $Z_k$ is the number of descendants of  $\tilde{z}$ starting at time $e^{\lambda k}$ at position $ae^{\frac{\lambda k}{1+\alpha}}$ over  $l_k:=e^{\lambda(k+1)}-e^{\lambda k}$ generations. The individuals of $\tilde{z}$  in generation $i$ are killed if they are out of the interval:
$I_i:=[(a-b)i^{\frac1{1+\alpha}},ai^{\frac1{1+\alpha}}]$.

%\subsection{{\blue An upper bound associated with $\mathbb{E}(Z_k^2)$}}
%~\\

For $   u,  v\in\mathcal{T}$, let $   u_j:=   u\wedge  v\in\mathcal{T}$  be the lowest common ancestor of them. We split $\mathbb{E}(Z_k^2)$ into the double sum over $   u,  v$ according to the generation $j$ as follows:

\begin{align}
\mathbb{E}(Z_k^2)=\mathbb{E}\Big(\sum_{   u>\tilde{z},  v>\tilde{z},|   u|=|  v|=e^{\lambda(k+1)}}
\mathbf{1}_{\{\forall e^{\lambda k}<i\leq e^{\lambda(k+1)},\,V(   u_i)\in I_i,\,V(  v_i)\in I_i\}}\Big)=\sum_{j=0}^{l_k}D_{k,j}, \nonumber
\end{align}
where $D_{k,l_k}=\mathbb{E}(Z_k)$ and for $j<l_k,$
\begin{align}
D_{k,j}:=&\mathbb{E}\Big(\sum_{   u>\tilde{z},|   u|=e^{\lambda(k+1)}}\mathbf{1}_{\{\forall e^{\lambda k}<i\leq e^{\lambda(k+1)},\,V(   u_i)\in I_i\}}\nonumber\\
& \sum_{  v>   u_{e^{\lambda k}+j},|  v|=e^{\lambda(k+1)},  v_{e^{\lambda k}+j+1}\neq   u_{e^{\lambda k}+j+1}}\mathbf{1}_{\{\forall e^{\lambda k}+j<i\leq e^{\lambda(k+1)},\,V(  v_i)\in I_i\}}\Big) .\nonumber
\end{align}
By Lemma \ref{L:2.1}, for $x\in I_{e^{\lambda k}+j}$ we have
\begin{align}
F_{k,j}(x):&=\mathbb{E}\Big(\sum_{  v>   u_{e^{\lambda k}+j},|  v|=e^{\lambda(k+1)}}
\mathbf{1}_{\{\forall\,e^{\lambda k}+j<i\leq e^{\lambda(k+1)},\,V(  v_i)\in I_i\}}\Big|V(   u_{e^{\lambda k}+j})=x\Big)\nonumber\\
&=\mathbb{E}\Big(e^{S_{l_k-j}}\mathbf{1}_{\{\forall\,0<i\leq l_k-j,\,x+S_i\in I_{e^{\lambda k}+j+i}\}}\Big)\label{E:2.34}\\
&\leq \exp\Big\{a e^\frac{\lambda (k+1)}{1+\alpha}-a(e^{\lambda k}+j)^\frac{1 }{1+\alpha}  +b(e^{\lambda k}+j)^{\frac1{1+\alpha}}\Big\}  \mathbb{P}\Big(\forall\,0<i\leq l_k-j,\,x+S_i\in I_{e^{\lambda k}+j+i}\Big).\nonumber
\end{align}

For some $R_k>0$ (Its value need to be determined),  we  define a processes $Z_k^{(k)}$ as follows.  If an individual has a number of children greater than $R_k$,  then we remove all the descendants of it.
We  add a superscript $^{(k)}$ when dealing with this new process $Z_k^{(k)}$. Clearly,   $Z_k^{(k)}\le Z_k$. Analogously to above discussion, we have
 \begin{align}
\mathbb{E}\Big((Z_k^{(k)})^2\Big)=\sum_{j=0}^{l_k}D_{k,j}^{(k)}. \label{E:2.36}
\end{align}
By \cite[Page 1002-1003]{j},
\begin{align}\label{bkj}
D_{k,j}^{(k)}\leq (R_k-1)\sup_{x\in I_{e^{\lambda k}+j+1}}F_{k,j+1}^{(k)}(x)\mathbb{E}(Z_k^{(k)}).
\end{align}
From the definition, it is not difficult to see that $F_{k,j+1}^{(k)}(x)\leq F_{k,j+1}(x)$.

Define $\beta(\rho,l):=\lfloor\rho l\rfloor + 1$, $\gamma(l):=l$ and write $j=\beta(\rho,l_k)-1$ for any $\rho\in(0,1)$.
Lemma \ref{L:2.4} yields that, uniformly in $\rho\in(0,1)$ and $x\in \mathcal{T}_{e^{\lambda k}+\beta(\rho,l_k)}$,
\begin{align}
\limsup_{k\rightarrow\infty}l_k^{-\frac1{1+\alpha}}
\log \mathbb{P}\big(\forall\;0<i\leq l_k-(j+1),\,x+S_i\in I_{e^{\lambda k} + j + 1 + i}\big)\leq -C_*\int^1_\rho\frac{1}{(g_2(t)-g_1(t))^\alpha}dt,\nonumber
\end{align}
where
\begin{align}
&g_2(t):=a\bigg(\Big(t+\frac{1}{e^\lambda-1}\Big)^{\frac1{1+\alpha}}-\Big(\frac{1}{e^\lambda-1}\Big)^{\frac1{1+\alpha}}\bigg), \label{g2}\\
&g(t):=b\Big(t+\frac{1}{e^\lambda-1}\Big)^{\frac1{1+\alpha}},\quad g_1(t)=g_2(t)-g(t).\label{g}
\end{align}
Combining with  \eqref{bkj} and \eqref{E:2.34}, we get that uniformly in $\rho\in(0,1)$,

 \begin{align}
& \limsup_{k\rightarrow\infty}  l_k^{-\frac1{1+\alpha}}
\log \frac{D^{(k)}_{k,\beta(\rho,l_k) - 1}}{\mathbb{E}(Z_k^{(k)})}\nonumber\\
&\leq   \limsup_{k\rightarrow\infty} l_k^{-\frac1{1+\alpha}}\log (R_k - 1) +g_2(1)-g_2(\rho)+g(\rho)-C_*\int_\rho^1\frac{1}{(g_2-g_1)^\alpha}. \label{E:2.35}
\end{align}

%\subsection{A lower bound for  $\mathbb{E}(Z_k^{(k)})$}
%~\\

For any $k\geq1,$ we consider i.i.d. random variable $X_i^{(k)},\;1\leq i\leq l_k$ with the same distribution as $X$ conditioned on $\upsilon\leq R_k$\,(with $(X,\upsilon)$ defined as in Lemma \ref{L:2.2}). Write
$S_j^{(k)}:=\sum_{i=1}^j X_i^{(k)}$ for any $0\leq j\leq l_k$. Let $\delta>0$ and $\varrho>0$ be the constants in \eqref{C:1.2}. For $\varepsilon>0$,
by Lemma \ref{L:2.2}, going along the line in \cite[section 4.3]{j}, we have
\begin{align}\label{zkk1}
\mathbb{E}(Z_k^{(k)})&=\mathbb{E}\Big(\sum_{  u>\tilde{z},\,|  u|=e^{\lambda(k+1)}}
\mathbf{1}_{\{\forall\,e^{\lambda k}<i\leq e^{\lambda(k+1)},\,V(  u_i)\in I_i,\,\#\Gamma(  u_{i-1})\leq R_k\}}\Big)\nonumber\\
&=\mathbb{E}\Big(e^{S_{l_k}}\mathbf{1}_{\{\forall\,i\leq l_k,\,ae^{\lambda k/(1+\alpha)}+S_i\in I_{e^{\lambda k}+i},\,\upsilon_i\leq R_k\}}\Big)\nonumber\\
&=\mathbb{E}\Big(e^{S_{l_k}}\mathbf{1}_{\{\forall\,i\leq l_k,\,ae^{\lambda k/(1+\alpha)}+S_i\in I_{e^{\lambda k}+i}\}}\Big|\upsilon_i\leq R_k,\,\forall\,i\leq l_k\Big)\mathbb{P}(\upsilon\leq R_k)^{l_k} \nonumber\\
&=\mathbb{P}(\upsilon\leq R_k)^{l_k}\mathbb{E}\Big(e^{S_{l_k}^{(k)}}\mathbf{1}_{\{\forall\,i\leq l_k,\,ae^{\lambda k/(1+\alpha)}+S_i^{(k)}\in I_{e^{\lambda k}+i}\}}\Big) \nonumber\\
& \geq \mathbb{P}(\upsilon\leq R_k)^{l_k}\exp\big\{l_k^{\frac1{1+\alpha}}\big(g_2(1)-\varepsilon\big)\big\}\nonumber\\
& \cdot\mathbb{P}\Big(
g_1(t)\leq\frac{S^{(k)}_{\lfloor tl_k\rfloor}}{\;l_k^{\frac1{1+\alpha}}}\leq g_2(t),t\in[0,1];\;S^{(k)}_{l_k}\geq l_k^{\frac1{1+\alpha}}(g_2(1)-\varepsilon)\Big),
\end{align}
Let $\delta>0$ and $\varrho>0$ be the constants in condition \eqref{C:1.2}. By H\"{o}lder's inequality,
\beqnn
\mathbb{P}(\upsilon>R_k)\leq \Big(\mathbb{E}(\#\mathcal{T}_1^{1+\delta})\Big)^\frac{\varrho}{1+\varrho} R_k^{-\frac{\delta\varrho}{1+\varrho}}
\left(\mathbb{E}\Big[\sum_{|   u|=1}e^{-(1+\varrho)V{(   u)}}\Big]\right)^{\frac{1}{1+\varrho}}.
\eeqnn
We now choose $R_k:=\lfloor e^{l_k^{1/c}}\rfloor$ for some $c>1+\alpha$. Therefore
\beqlb\label{3.6a}
\lim_{k\rightarrow\infty}  l_k^{-\frac1{1+\alpha}}\log \big(\mathbb{P}(\upsilon\leq R_k)^{l_k}\big)=0.
\eeqlb
By the Markov property and using the notation of Lemma \ref{L:2.11}, there exist $M,\varepsilon_1>0$ such that for sufficiently large $k$ and any small $\varepsilon_2>0$,
\begin{align}\label{3.7}
&\mathbb{P}(g_1(t)\leq  l_k^{-\frac1{1+\alpha}}S^{(k)}_{\lfloor tl_k\rfloor}\leq g_2(t),t\in[0,1];\,S^{(k)}_{l_k}\geq l_k^{\frac1{1+\alpha}}(g_2(1)-\varepsilon))\nonumber \\
&\geq P_{l_k}(M,\varepsilon_1,\varepsilon_2)\inf_{-M\varepsilon_2l_k^{\frac1{1+\alpha}}\leq z\leq -\varepsilon_1\varepsilon_2l_k^{\frac1{1+\alpha}}}H_{l_k}(z,\varepsilon_2,g_1,g_2),
\end{align}
where
\begin{align}
H_{l_k}(z,\varepsilon_2,g_1,g_2):=&\mathbb{P}_z\Big(
S^{(k)}_{l_k-\lfloor\varepsilon_2l_k^{\frac1{1+\alpha}}\rfloor}\geq l_k^{\frac1{1+\alpha}}(g_2(1)-\varepsilon),\,\forall\,i\leq{l_k-\lfloor\varepsilon_2l_k^{\frac1{1+\alpha}}\rfloor} , \nonumber\\
&g_1\Big(\frac{\,\lfloor\varepsilon_2l_k^{\frac1{1+\alpha}}\rfloor+i}{l_k}\Big)\leq \frac{S_i^{(k)}}{l_k^{\frac1{1+\alpha}}}\leq g_2\Big(\frac{\,\lfloor\varepsilon_2l_k^{\frac1{1+\alpha}}\rfloor+i}{l_k}\Big)
\Big). \nonumber
\end{align}
By  Lemma \ref{L:2.8}, we get that
\begin{align}\label{3.8}
\liminf_{k\rightarrow\infty} l_k^{-\frac1{1+\alpha}}
\log \inf_{-M\varepsilon_2l_k^{\frac1{1+\alpha}}\leq z\leq -\varepsilon_1\varepsilon_2l_k^{\frac1{1+\alpha}}}H_{l_k}(z,\varepsilon_2,g_1,g_2) \geq-C_*\int_0^1\frac{1}{(g_2-g_1)^\alpha}.
\end{align}
  Then put  (\ref{3.6a})--(\ref{3.8}) into (\ref{zkk1}). Letting $\varepsilon\rightarrow0$ and recalling Lemma~\ref{L:2.11} (and many to one lemma), we arrive at
\beqlb\label{zkk}
\liminf_{k\rightarrow\infty} l_k^{-\frac1{1+\alpha}}\log \mathbb{E}(Z^{(k)}_k)
\geq g_2(1)- C_*\int_0^1\frac{1}{(g_2-g_1)^\alpha}.
\eeqlb
~\\

Now we have
\begin{lemma}\label{L:4.1}
Choose $\lambda$ sufficiently large such that $e^\lambda\in\mathbb{N}^*$. For fixed $\theta\in (0,1)$, set $\nu_k=\theta\mathbb{E} Z_k^{(k)}$   and define
$$ T_k=\mathbb{P}(Z_k^{(k)}\geq  \nu_k). $$
 If $a>a_\alpha$, then
$$
\sum_{k=0}^\infty e^{-\nu_{k}T_{k+1}}<+\infty.
$$
\end{lemma}

{\bf Proof.} The proof is similar to that of \cite[Lemma 4.1]{j}, which is for the finite variance case.

Combining (\ref{zkk}) with \eqref{E:2.35} yields that, uniformly in $\rho\in(0,1)$,
  $$
\limsup_{k\rightarrow\infty}  l_k^{-\frac1{1+\alpha}}\log \frac{D^{(k)}_{k,\beta(\rho,l_k)-1}}{\big(\mathbb{E}(Z_k^{(k)})\big)^2}\leq -g_2(\rho)+g(\rho)+C_*\int_0^\rho\frac{1}{(g_2-g_1)^\alpha}.
$$
Together with \eqref{E:2.36} and  the Paley-Zygmund inequality
\begin{align}
T_k\geq (1-\theta)^2\frac{\;\,(\mathbb{E}(Z^{(k)}_k))^2}{\mathbb{E}(Z^{(k)}_k)^2},  \label{E:2.37}
\end{align}
we have that
\begin{align}
\liminf_{k\rightarrow\infty} l_k^{-\frac1{1+\alpha}}
\log T_k\geq \min_{0\leq \rho\leq1}\left\{g_2(\rho)-g(\rho)-C_*\int_0^\rho\frac{1}{(g_2-g_1)^\alpha}\right\}.
\label{zkk_2}
\end{align}
{Define}$$
G_\lambda(\rho):=-g_2(\rho)+g(\rho)+C_*\int_0^\rho\frac{dt}{g(t)^\alpha}+
e^{-\lambda/(1+\alpha)}\left(-g_2(1)+C_*\int_0^1\frac{dt}{g(t)^\alpha}\right).
$$
Denote
$f(t)=(t+\frac{1}{e^\lambda-1})^{\frac1{1+\alpha}}$ for $t\in[0,1]$.
 By (\ref{g2}) and (\ref{g}) we have $g_2=af-af(0)$ and $g=bf$. Then
\begin{align}
G_\lambda(\rho)=& af(0)+(b-a)f(\rho)+\frac{C_*}{b^\alpha}\int_0^\alpha\frac{dt}{f(t)^\alpha}\nonumber\\
&+e^{-\lambda/(1+\alpha)}\left(af(0)-af(1)+\frac{C_*}{b^\alpha}\int_0^1\frac{dt}{f(t)^\alpha}\right).\nonumber
\end{align}
Noting that $f(1)=e^{\lambda/(1+\alpha)}f(0)$ and $f'=\frac{1}{1+\alpha}f^{-\alpha}$, we  have
$$
G_\lambda(\rho)=\Big(b+\frac{(1+\alpha)C_*}{b^\alpha}-a\Big)f(\rho)+e^{-\lambda/(1+\alpha)}
\left(af(0)-\frac{(1+\alpha)C_*}{b^\alpha}f(0)\right).
$$
Since $a>a_\alpha$, we can choose $b$ such that $b+\frac{(1+\alpha)C_*}{b^\alpha}<a$. For this $b$, noticing that  $f$ is increasing on $[0,1]$, we obtain
\beqnn
\max_{0\leq\rho\leq1}G_\lambda(\rho)&=& G_\lambda(0)\\
&=& f(0)\Big[\Big(b+\frac{(1+\alpha)C_*}{b^\alpha}-a\Big)
+e^{-\lambda/(1+\alpha)}\Big(a-\frac{(1+\alpha)C_*}{b^\alpha}\Big)\Big] \\
&<& 0,
\eeqnn
for sufficiently large $\lambda$.
Meanwhile,  $$ g_2(1)-C_*\int_0^1\frac{dt}{g(t)^\alpha}= \big(f(1)-f(0)\big)\Big(a-\frac{(1+\alpha)C_*}{b^\alpha}\Big) >0.$$
Then for sufficiently large $\lambda$, \beqnn
A&:= &\min_{0\leq\rho\leq1}\left( g_2(\rho)-g(\rho)-C_*\int_0^\rho\frac{dt}{g(t)^\alpha}
 +g_2(1)-C_*\int_0^1\frac{dt}{g(t)^\alpha}\right) \\
 &= & \min_{0\leq\rho\leq1}(-G_\lambda(\rho))+(1-e^{-\frac{\lambda}{1+\alpha}})\left(g_2(1)-C_*\int_0^1\frac{dt}{g(t)^\alpha}\right) \\ &>&0.\eeqnn
 This  together with \eqref{zkk}--\eqref{zkk_2}, yields that for sufficiently large $k$ (noting that $l_{k+1}>l_k$),  $$  \nu_k T_{k+1}  \ge \theta \exp\{Al_k^{\frac1{1+\alpha}}\}. $$
The proof is concluded.
\qed
~\\
}

%\noindent\textbf{Proposition 1.4} {\em
%If $a>a_\alpha$, then the equation $a=b+\frac{ 1+\alpha }{b^\alpha}C_*$ have two solutions in $b$. Let $r_a$ be the larger solution, i.e., $r_a>\frac{\alpha a_\alpha}{ 1+\alpha }$. For any $\varepsilon>0$, for any $N\in \mathbb{N}$ large enough, we have with positive probability:
%$$
%\forall\,\;k\geq1,\;\;\#\{   u\in\mathcal{T}_{N^k}:\,\forall\,i\leq N^k,\,(a-r_a)i^{\frac{1}{1+\alpha}}\leq V(   u_i)\leq ai^{\frac{1}{1+\alpha}}\}\geq
%\exp(N^{k/(1+\alpha)}(r_a-\varepsilon)).
%$$
%}
{\bf Proof of Proposition \ref{P:1.3}.} Suppose that $Z_k^{(k)}$ , $Z_k$ and $\nu_k$ are defined as before. For any $n\geq1$, define
$$
P_n:=\mathbb{P}\Big(\,\forall\,1\leq k\leq n,\;\#\{u\in\mathcal{T}_{e^{\lambda k}}:\,\forall\,i\leq e^{\lambda k},\,(a-r_a)i^{\frac1{1+\alpha}}\le V(u_i)\leq ai^{\frac1{1+\alpha}}\}\geq \nu_{k - 1}\Big).$$

If $1\leq n_0\leq n$, by the Markov property and independence of individuals in generation $e^{\lambda k}$,   we have
$$
P_{n+1}\geq P_n(1-(1-\mathbb{P}(Z_n\ge \nu_n) )^{\nu_{n-1}}).$$
Observe that $Z_k^{(k)}\le Z_k$. We have
$
\mathbb{P}(Z_n^{(n)}\ge \nu_n)\le \mathbb{P}(Z_n\ge \nu_n)
$
and $$P_{n+1}\geq P_n(1-(1-T_n )^{\nu_{n-1}}).$$
By induction, we obtain
$$
P_n\geq P_{n_0}\prod_{k=n_0}^{n-1}\big(1-(1-T_k)^{\nu_{k-1}}\big)\geq P_{n_0}\prod_{k=n_0}^{n-1}\big(1-e^{-\nu_{k-1}T_k}\big),\quad n>n_0.$$
\beqlb\label{Pn}
\log P_n\geq \log P_{n_0}+\sum_{k=n_0}^n \log(1-e^{-\nu_{k-1}T_k}).
\eeqlb
Applying $\log(1+x)\sim x (x\to 0^+)$, by Lemma~\ref{L:4.1}, we have
$\sum_{k=n_0}^\infty \log(1-e^{-\nu_{k-1}T_k})>-\infty$. Then there exists $p>0$ such that for all sufficiently large $n$, we have  $P_{n}\ge p>0$. By \eqref{zkk} and recalling $l_k:=e^{\lambda(k+1)}-e^{\lambda k}$, we have for sufficiently large $\lambda$ and $k$,
\begin{align}
\nu_k=\theta\mathbb{E}(Z^{(k)}_k)&\geq \theta\exp\Big\{l_k^{1/(1+\alpha)}\Big(g_2(1)- C_*\int_0^1\frac{1}{(g_2-g_1)^\alpha}\Big)\Big\}\nonumber\\
&\geq\theta\exp\Big\{l_k^{1/(1+\alpha)}(r_a-\varepsilon) \Big\}\nonumber\\
&\geq\theta\exp\{(1-e^{-\lambda})^{\frac{1}{1+\alpha}}\cdot e^{\lambda (k+1)/(1+\alpha)}(r_a-\varepsilon)\} \nonumber\\
&\ge \exp\big\{\frac{1}{2}N^{ \frac{k+1}{1+\alpha} }(r_a-\varepsilon)\},
\end{align}
by choosing large $\lambda$ such  that $N=e^\lambda\in\mathbb{N^*}$ and $1-e^{-\lambda}>1/2$.
 Consequently,
\beqnn
  \mathbb{P}(B_N)&=& \mathbb{P}\Big(\,\forall\,k\ge 1,\#\{u\in\mathcal{T}_{N^k}:\,\forall\,i\leq N^k,(a-r_a)i^{\frac1{1+\alpha}}\le V(u_i)\leq ai^{\frac1{1+\alpha}}\}\geq \exp\big\{\frac{1}{2}N^{ \frac{k}{1+\alpha} }(r_a-\varepsilon)\} \Big)\nonumber\\
&\ge&\mathbb{P}\Big(\,\forall\,k\ge 1,\#\{u\in\mathcal{T}_{N^k}:\,\forall\,i\leq N^k,(a-r_a)i^{\frac1{1+\alpha}}\le V(u_i)\leq ai^{\frac1{1+\alpha}}\}\geq \nu_{k-1} \Big)\nonumber\\
&=& \lim_n P_n >0.\eeqnn
\qed

 {\bf Proof of the lower bound of Theorem~\ref{T:1.1}.} The proof is immediate by Proposition \ref{P:1.3}.   \qed

\section{Upper bound for the survival probability}

 The idea and technical route of the upper bound are similar to \cite[Section 3.4]{j} (which is for the cases $\alpha=2$). We only explain  the sketch of the proofs and omit the  details.

Fix $a>0$. Clearly,
\beqlb\label{last}
\mathbb{P}\big(\exists\,   u\in\mathcal{T}_\infty,\,\forall\,i,\,V(   u_i)\leq a i^{\frac1{1+\alpha}}\big)=\lim_{n\rightarrow\infty}\mathbb{P}\big(\exists\,   u\in\mathcal{T}_n,\,\forall\,i\leq n,\,V(   u_i)\leq a i^{\frac1{1+\alpha}}\big).
\eeqlb

  Let $h$ be
  some continuous function from $[0,1]$ to $[0,+\infty)$.

\begin{lemma}\label{Lemma4.2} For $a\in(0,a_\alpha)$, we have
\begin{align}\label{VV}
\limsup_{n\rightarrow\infty}n^{-\frac1{1+\alpha}}\log \mathbb{P}\big(\exists\,   u\in\mathcal{T}_n,\,\forall\,i\leq n,\,V(   u_i)\leq ai^{\frac1{1+\alpha}}\big)\leq -K,
\end{align}
where $K:=\min(K_1,K_2)$, and \begin{align} & K_1:=-a+C_*\int_0^1\frac{dt}{h (t)^\alpha},\nonumber\\
& K_2:=\min_{0\leq\rho\leq1}\Big\{- a\rho^{\frac1{1+\alpha}}+h(\rho) + C_*\int_0^\rho \frac{dt}{h(t)^\alpha}\Big\}\label{E:3.2}
\end{align}
for some non-negative continuous function $h$ on $[0,1]$.
\end{lemma}
{\bf Proof.} The proof    can be obtained by the method of \cite[sections 3.1-3.3]{j}, if we replace $aj^{1/3}$, $s_2$ and $s$ therein  by $aj^{\frac 1{1+\alpha}}$, $K_2$ and $K$, respectively, and apply our
Lemma~\ref{L:2.9} and Lemma~\ref{L:2.4} (instead of Lemma 2.4 and Proposition 2.5 in \cite{j}), with $g(t)=at^{\frac1{1+\alpha}}$
and $f(t)=at^{\frac1{1+\alpha}}-h(t)$.  We omit the details here.\qed

Set $a\in(0,a_\alpha)$. With the help of Lemma~\ref{Lemma4.2}, if we can find a function $h$  such that $K>0$, then   the proof of the upper bound of Theorem \ref{T:1.1} is completed. In the following we do this work.

Add the constraint $h(1)=0$ (but assume $\int_0^1\frac{dt}{h(t)^\alpha}<\infty$). Taking $\rho=1,$ we see that $K_2\leq K_1$. As a result, $K=K_2$. If we can choose $h$ in such a way that  $h(0)>0$ and $-a\rho^{\frac1{1+\alpha}}+h(\rho)+C_*\int_0^\rho\frac{d   u}{h(   u)^\alpha}$ does not depend on $\rho$, then by \eqref{E:3.2},  $K=K_2\equiv h(0)$. In this case, $h$ should be  the solution of the equation:
\begin{align}\label{E:3.3}
\forall\,t\in[0,1],\;\;\;-at^{\frac1{1+\alpha}}+h(t)+C_*\int_0^t\frac{dx}{h(x)^\alpha}=K,
\end{align}
where $K$ is some positive constant, the value of which is to be set later in such a way that $h(1-)=0$. According to the discussion above, this value of $K$ will give a bound for the rate of decay of the survival probability.

Equivalently, equation \eqref{E:3.3} may be written as $h(0)=K$ and $\forall\,t\in(0,1)$,
\begin{align}                            \label{E:4.1}
h'(t)=\frac{a}{\,1+\alpha}t^{-\!\frac{\alpha}{1+\alpha}}-\frac{C_*}{h(t)^\alpha}.
\end{align}

By the Picard-Lindel\"{o}f theorem, this ordinary equation admits a unique maximal solution $h$ defined on an interval $[0,t_{\max})$.  Actually, as \cite[Proposition 3.6]{j}, we now have

\begin{proposition}\label{P:3.6}
Let $h$ be the unique maximal solution of equation \eqref{E:4.1} with initial condition $h(0)=1$.
If $a<a_\alpha$, then $t_{\text{max}}<+\infty$ and $h(t)\rightarrow 0$ as $t\rightarrow t_{\text{max}}$.

\end{proposition}

{\bf Proof of the upper bound of Theorem~\ref{T:1.1}.} For $a<a_\alpha$, suppose that $h$ is  the unique maximal solution of equation \eqref{E:4.1} with initial condition $h(0)=1$. By Proposition \ref{P:3.6}, $t_{\text{max}}\in (0,\infty)$.   Define $\epsilon=1/t_{\text{max}}$ and $h_\epsilon(t)=\epsilon^{-1/(1+\alpha)}h(\epsilon t)$. Direct calculation yields that $h_\epsilon$ is the solution of equation \eqref{E:4.1} on $[0,1)$ with initial condition $h_\epsilon(0)=\epsilon^{-1/(1+\alpha)}$.  Choosing  $K=h_\epsilon(0)=\epsilon^{-1/(1+\alpha)}$and applying Lemma \ref{Lemma4.2} in (\ref{last}), we obtain the desired result.

 \qed

\end{document}